\documentclass[12pt]{amsart}
\usepackage{eucal,times,amsmath,amsthm,amssymb,mathrsfs,stmaryrd,color,enumerate,accents}
\usepackage[all]{xy}
\usepackage{url}
\usepackage{fullpage}

\usepackage{float}
\usepackage{graphicx}
\usepackage{amscd}
\usepackage{amsthm}
\usepackage{amsxtra}
\usepackage{verbatim}
\usepackage{endnotes}
\usepackage{times}
\usepackage{mathtools}
\usepackage[colorlinks=true,hyperindex]{hyperref}

\usepackage[T2A,T1]{fontenc}
\DeclareSymbolFont{cyrillic}{T2A}{cmr}{m}{n}
\DeclareMathSymbol{\Sha}{\mathalpha}{cyrillic}{216}

\theoremstyle{plain}
\newtheorem{X}{X} 
\newtheorem{theorem}[X]{Theorem}

\newtheorem{lemma}[X]{Lemma}

\theoremstyle{definition}
\newtheorem{definition}[X]{Definition}

\newtheorem{remark}[X]{Remark}

\newtheorem{question}[X]{Question}

\begin{document}

\title{A note on the Bloch-Tamagawa space and Selmer groups \emph{}}
\address{Department of Mathematics, University of Maryland, College Park, MD 20742 USA.} 
\email{atma@math.umd.edu}
\urladdr{\href{http://www2.math.umd.edu/~atma}{http://www2.math.umd.edu/~atma}}
\author{Niranjan Ramachandran}

\begin{abstract} For any abelian variety $A$ over a number field, we construct an extension of the Tate-Shafarevich group by the Bloch-Tamagawa space using the recent work of Lichtenbaum and Flach. This gives a new example of a Zagier sequence for the Selmer group of $A$. 
 \end{abstract}

\maketitle

\subsection*{Introduction.}  Let $A$ be an abelian variety over a number field $F$ and $A^{\vee}$ its dual.  B.~Birch and P.~Swinnerton-Dyer, interested in defining the Tamagawa number $\tau(A)$ of $A$, were led to their celebrated conjecture \cite[Conjecture 0.2]{bloch} for the L-function $L(A,s)$ (of $A$ over $F$) which predicts both its order $r$ of vanishing and its leading term $c_A$ at $s=1$.  The difficulty in defining $\tau(A)$ directly is that the adelic quotient $\frac{A(\mathbb A_F)}{A(F)}$ is Hausdorff only when $r=0$, i.e., $A(F)$ is finite. S.~Bloch \cite{bloch} has introduced a semiabelian variety $G$ over $F$ with quotient $A$ such that $G(F)$ is discrete and cocompact in $G(\mathbb A_F)$ \cite[Theorem 1.10]{bloch} and famously proved \cite[Theorem 1.17]{bloch} that the Tamagawa number conjecture - recalled briefly below, see (\ref{tam})  -  for $G$ is equivalent to the Birch-Swinnerton-Dyer conjecture for $A$ over $F$. Observe that $G$ is not a linear algebraic group.  The Bloch-Tamagawa space $X_A = \frac{G(\mathbb A_F)}{G(F)}$ of $A/F$ is compact and Hausdorff.

The aim of this short note is to indicate a functorial construction of a locally compact group $Y_A$
\begin{equation}\label{ext}0 \to X_A \to Y_A \to \Sha(A/F) \to 0,\end{equation} an extension of the Tate-Shafarevich group $\Sha(A/F)$ by $X_A$. The compactness of $Y_A$ is clearly equivalent to the finiteness of $\Sha(A/F)$. This construction would be straightforward if $G(L)$ were discrete in $G(\mathbb A_L)$ for all finite extensions $L$ of $F$. But this is not true (Lemma \ref{hff}): the quotient {\it $$\frac{G(\mathbb A_L)}{G(L)}$$ is not Hausdorff, in general.}

  The very simple idea for the construction of $Y_A$ is: {\it Yoneda's lemma}. Namely, we consider the category of topological $\mathcal G$-modules as a subcategory of the classifying topos $B\mathcal G$ of $\mathcal G$ (natural from the context of  the continuous cohomology of a topological group $\mathcal G$, as in S.~Lichtenbaum \cite{sl1}, M.~Flach \cite{flach}) and construct $Y_A$ via the classifying topos of the Galois group of $F$.

 D.~Zagier \cite{zagier} has pointed out that the Selmer groups ${\rm Sel}_m(A/F)$ (\ref{selm}) can be obtained from certain two-extensions (\ref{don}) of $\Sha(A/F)$ by $A(F)$; these we call Zagier sequences. We show how $Y_A$ provides a new natural  Zagier sequence. In particular, this shows that $Y_A$ is not a split sequence.

Bloch's construction has been generalized to one-motives; it led to the Bloch-Kato conjecture on Tamagawa numbers of motives \cite{bk}; it is close in spirit to Scholl's method of relating non-critical values of L-functions of pure motives to critical values of L-functions of mixed motives \cite[p.~252]{sl2} \cite{scholl2, scholl1}. 
 
\subsection*{Notations.} We write $\mathbb A = \mathbb A_f \times \mathbb R$ for the ring of adeles over $\mathbb Q$; here $\mathbb A_f = \hat{\mathbb Z}\otimes_{\mathbb Z} \mathbb Q$ is the ring of finite adeles. For any number field $K$, we let $\mathcal O_K$ be the ring of integers, $\mathbb A_K $ denote the ring of adeles $\mathbb A \otimes_{\mathbb Q}K$ over $K$; write $\mathbb I_K$ for the ideles. Let $\bar{F}$ be a fixed algebraic closure of $F$ and write $\Gamma = \text{Gal}(\bar{F}/F)$ for the Galois group of $F$. For any abelian group $P$ and any integer $m >0$, we write $P_m$ for the $m$-torsion subgroup of $P$.  A topological abelian group is Hausdorff.

\subsection*{Construction of $Y_A$.} This will use the continuous cohomology of $\Gamma$ via classifying spaces as in  \cite{sl1, flach} to which we refer for a detailed exposition.

For each field $L$ with $F \subset L \subset \bar{F}$, the group $G(\mathbb A_L)$ is a locally compact group. If $L/F$ is Galois, then $$G(\mathbb A_L)^{{\rm Gal}(L/F)} = G(\mathbb A_F).$$  So  $$\mathbb E = \underset{\rightarrow}{\rm lim} ~G(\mathbb A_L),$$ the direct limit of locally compact abelian groups, is equipped with a continuous action of $\Gamma$. The natural map \begin{equation}\label{EE} E:= G(\bar{F}) \hookrightarrow \mathbb E\end{equation} is $\Gamma$-equivariant.
Though the subgroup $G(F) \subset G(\mathbb A_F)$ is discrete, the subgroup $$E \subset \mathbb E$$ fails to be discrete; this failure happens at finite level (see Lemma \ref{hff} below). The non-Hausdorff nature of the quotient $${\mathbb E}/{E}$$ directs us to consider the classifying space/topos. 

Let $Top$ be the site defined by the category of (locally compact) Hausdorff topological spaces with the open covering Grothendieck topology (as in the "gros topos" of \cite[\S2]{flach}). Any locally compact abelian group $M$ defines a sheaf $yM$ of abelian groups on $Top$; this (Yoneda) provides a fully faithful embedding of the (additive, but not abelian) category $Tab$ of locally compact abelian groups into the (abelian) category $\mathcal  Tab$ of sheaves of abelian groups on $Top$. Write $\mathcal Top$ for the category of sheaves of sets on $Top$ and let $y: Top \to \mathcal Top$ be the Yoneda embedding. A generalized topology on a given set $S$ is an object $F$ of $\mathcal Top$ with $F({*}) =S$.  
 
For  any (locally compact) topological group $\mathcal G$, its classifying topos $B\mathcal G$ is the category of objects $F$ of $\mathcal Top$ together with an action $y\mathcal G \times F \to F$. An abelian group object $F$ of $B\mathcal G$ is a sheaf on $\mathcal Top$, together with actions  $y\mathcal G(U) \times F(U) \to F(U)$, functorial in $U$; we write $H^i(\mathcal G, F)$ (objects of $\mathcal Tab$) for the continuous/topological group cohomology of $\mathcal G$ with coefficients in $F$. These arise from the global section functor $$B\mathcal G \to \mathcal Tab, \qquad F \mapsto F^{y\mathcal G}.$$
Details for the following facts can be found in \cite[\S3]{flach} and \cite{sl1}. 
\begin{enumerate} 
\item[(a)] (Yoneda) Any topological $\mathcal G$-module $M$ provides an (abelian group) object $yM$ of $B\mathcal G$; see \cite[Proposition 1.1]{sl1}. 
\item[(b)] If $0\to M \to N$ is a map of topological $\mathcal G$-modules with $M$ homeomorphic to its image in $N$, then the induced map $yM \to yN$ is injective \cite[Lemma 4]{flach}.
\item[(c)] Applying Propositions 5.1 and 9.4 of \cite{flach} to the profinite group $\Gamma$ and any continuous $\Gamma$-module $M$ provide an isomorphism $$H^i(\Gamma, yM) \simeq yH^i_{cts}(\Gamma, M)$$ between this topological group cohomology and the continuous cohomology (computed via continuous cochains). This is also proved in \cite[Corollary 2.4]{sl1}.
\end{enumerate} 

 For any continuous homomorphism $f:M \to N$ of topological abelian groups, the cokernel of $yf: yM \to yN$ is well-defined in $\mathcal Tab$ even if the cokernel of $f$ does not exist in $Tab$. If $f$ is a map of topological $\mathcal G$-modules, then the cokernel of the induced map $yf: yM \to yN$, a well-defined abelian group object of $B\mathcal G$, need not be of the form $yP$.

By (a) and (b) above, the pair of topological $\Gamma$-modules $E \hookrightarrow \mathbb E$ (\ref{EE}) gives rise to a pair $yE \hookrightarrow y\mathbb E$  of objects of $B\Gamma$. Write $\mathcal Y$ for the quotient object $\frac{y\mathbb E}{yE}$.
As $\mathbb E/{E}$ is not Hausdorff (Lemma \ref{hff}), $\mathcal Y$ is not $yN$ for any topological $\Gamma$-module $N$.




\begin{definition}We set $\mathcal Y_A = H^0(\Gamma, \mathcal Y)\in \mathcal Tab$.\end{definition} 

Our main result is the 
\begin{theorem}\label{main}  

(i) $\mathcal Y_A$ is the Yoneda image $yY_A$ of a Hausdorff locally compact topological abelian group $Y_A$. 

(ii)  $X_A$ is an open subgroup of $Y_A$. 

(iii) The group $Y_A$ is compact if and only if $\Sha(A/F)$ is finite. If $Y_A$ is compact, then the index of $X_A$ in  $Y_A$ is equal to $\#\Sha(A/F)$.  
\end{theorem} 

 As $\Sha(A/F)$ is a torsion discrete group, the topology of $Y_A$ is determined by that of $X_A$. 


\begin{proof} (of Theorem \ref{main}) The basic point is the proof of (iii).  From the exact sequence $$0 \to yE \to y\mathbb E \to \mathcal Y \to 0$$ of abelian objects  in $B\Gamma$, we get a long exact sequence (in $\mathcal Tab$) 
\begin{align*} 0 \to H^0(\Gamma, yE) \to H^0(\Gamma, y\mathbb E) \to & \\ 
\to H^0(\Gamma, \mathcal Y) \to H^1(\Gamma, yE) \xrightarrow{j} H^1(\Gamma, y\mathbb E) & \to \cdots .\end{align*}

We have the following identities of topological groups: $H^0(\Gamma, yE) = yG(F)$ and $H^0(\Gamma, y\mathbb E) = yG(\mathbb A_F)$, and by \cite[Lemma 4]{flach}, $\frac{yG(\mathbb A_F)}{yG(F)} \simeq yX_A$.  This exhibits $yX_A$ as a sub-object  of $\mathcal Y_A$ and provides the exact sequence
$$ 0 \to yX_A \to \mathcal Y_A \to {\rm Ker}(j) \to 0.$$  If ${\rm Ker}(j) =y\Sha(A/F)$, then $\mathcal  Y_A = yY_A$ for a unique topological abelian group $Y_A$ because $\Sha(A/F)$ is a torsion discrete group. Thus, it suffices to identify ${\rm Ker}(j)$ as $y\Sha(A/F)$. 
Let  $\mathbb E^{\delta}$ denote $\mathbb E$ endowed with the discrete topology; the identity map on the underlying set provides a continuous $\Gamma$-equivariant map $\mathbb E^{\delta} \to \mathbb E$. Since $E$ is a discrete $\Gamma$-module, the inclusion $E \to \mathbb E$ factorizes via $\mathbb E^{\delta}$. By item (c) above, ${\rm Ker}~(j)$ is isomorphic to the Yoneda image of the kernel of the composite map $$H^1_{cts}(\Gamma, E) \xrightarrow{j'} H^1_{cts}(\Gamma, \mathbb E^{\delta}) \xrightarrow{k} H^1_{cts}(\Gamma, \mathbb E).$$ Since $E$ and $\mathbb E^{\delta}$ are discrete $\Gamma$-modules,  the map $j'$ identifies with the map of ordinary Galois cohomology groups $$H^1(\Gamma, E) \xrightarrow{j"} H^1(\Gamma, \mathbb E^{\delta}).$$The traditional definition \cite[Lemma 1.16]{bloch} of $\Sha(G/F)$ is as ${\rm Ker}(j")$.   As $$\Sha(A/F) \simeq \Sha(G/F)$$ \cite[Lemma 1.16]{bloch}, to prove Theorem \ref{main}, all that remains is the injectivity of $k$.  This is straightforward from the standard description of $H^1$ in terms of crossed homomorphisms: if $f: \Gamma \to \mathbb E^{\delta}$ is a crossed homomorphism with $kf$ principal, then there exists $\alpha \in \mathbb E$ with $f: \Gamma \to \mathbb E$ satisfies $$f(\gamma) = \gamma (\alpha) -\alpha \quad \gamma \in \Gamma.$$ This identity clearly holds in both $\mathbb E$ and $\mathbb E^{\delta}$. Since the $\Gamma$-orbit of any element of $\mathbb E$ is finite, the left hand side is a continuous map from $\Gamma$ to $\mathbb E^{\delta}$. Thus, $f$ is already a principal crossed (continuous) homomorphism. So $k$ is injective,  finishing the proof of Theorem \ref{main}. \end{proof} 
 
 \begin{remark} The proof above shows:  If the stabilizer of every element of a topological $\Gamma$-module $N$ is open in $\Gamma$, then the natural map $H^1(\Gamma, N^{\delta}) \to H^1(\Gamma, N)$ is injective. \end{remark} 
\subsection*{Bloch's semi-abelian variety $G$}  \cite{bloch, MR693318}

Write $A^{\vee}(F) = B \times \text{finite}$. By the Weil-Barsotti formula,  $${\rm Ext}^1_F(A, \mathbb G_m) \simeq  A^{\vee}(F).$$ Every point $P \in A^{\vee}(F)$ determines a semi-abelian variety $G_P$ which is an extension of $A$ by $\mathbb G_m$. Let $G$ be the semiabelian variety determined by $B$: 

\begin{equation}\label{semi}0 \to T \to G \to A \to 0,\end{equation} 
an extension of $A$ by the torus $T =Hom(B, \mathbb G_m)$.  
The semiabelian variety $G$ is the Cartier dual \cite[\S10]{pd} of the one-motive $$[B \to A^{\vee}].$$ 
The sequence (\ref{semi}) provides (via Hilbert Theorem 90) \cite[(1.4)]{bloch} the following exact sequence
\begin{equation}\label{myst}
0 \to \frac{T(\mathbb A_F)}{T(F)} \to \frac{G(\mathbb A_F)}{G(F)} \to \frac{A(\mathbb A_F)}{A(F)} \to 0. 
\end{equation} It is worthwhile to contemplate this mysterious sequence: the first term is a Hausdorff, non-compact group and the last is a compact  non-Hausdorff group, but the middle term is a compact Hausdorff group!

\begin{lemma}\label{hff} Consider a field $L$ with $F \subset L \subset \bar{F}$. The group $G(L)$ is a discrete subgroup of $G(\mathbb A_L)$ if and only if $A(K)\subset A(L)$ is of finite index.  
\end{lemma}

\begin{proof} Pick a subgroup $C\simeq \mathbb Z^s$ of $A^{\vee}(L)$ such that $B \times C$ has finite index in $A^{\vee}(L)$. The Bloch semiabelian variety $G'$ over $L$ determined by $B \times C$ is an extension of $A$ by $T' = Hom(B \times C, \mathbb G_m)$. One has an exact sequence $ 0\to T" \to G' \to G \to 0$ defined over $L$ where $T"= Hom(C, \mathbb G_m)$ is a split torus of dimension $s$. Consider the commutative diagram with exact rows and columns

$
\begin{CD}
{} @. 0 @. 0@. {} \\
@. @VVV @VVV @. @.\\
{} @.  \frac{T"(\mathbb A_L)}{T"(L)} @=\frac{T"(\mathbb A_L)}{T"(L)} @. {} \\  
@. @VVV @VVV @. @.\\
0 @>>> \frac{T'(\mathbb A_L)}{T'(L)} @>>> \frac{G'(\mathbb A_L)}{G'(L)} @>>> \frac{A(\mathbb A_L)}{A(L)}\\
@. @VVV @VVV @| \\ 
0 @>>> \frac{T(\mathbb A_L)}{T(L)} @>>> \frac{G(\mathbb A_L)}{G(L)} @>>> \frac{A(\mathbb A_L)}{A(L)}\\
@. @VVV @VVV @VVV\\
{} @. 0 @. 0 @. 0.\\ 
\end{CD}
$

$${} $$ 
 The proof of surjectivity in the columns follows Hilbert Theorem 90 applied to $T"$ \cite[(1.4]{bloch}). The Bloch-Tamagawa space $X'_A = \frac{G'(\mathbb A_L)}{G'(L)}$ for $A$ over $L$ is compact and Hausdorff; its quotient by $$\frac{T"(\mathbb A_L)}{T"(L)} = (\frac{\mathbb I_L}{L^*})^s$$ is $\frac{G(\mathbb A_L)}{G(L)}.$ The quotient is Hausdorff if and only if $s =0$.\end{proof} 

A more general form of Lemma \ref{hff} is implicit in \cite{bloch}: For any one-motive $[N \xrightarrow{\phi} A^{\vee}]$ over $F$, write $V$ for its Cartier dual (a semiabelian variety), and  put $$X = \frac{V(\mathbb A_F)}{V(F)}.$$ We assume that the $\Gamma$-action on $N$ is trivial. Then $X$ is compact if and only if ${\rm Ker}(\phi)$ is finite; $X$ is Hausdorff if and only if the image of $\phi$ has finite index in $A^{\vee}(F)$. 

\subsection*{Tamagawa numbers.} 
Let $H$ be  a semisimple algebraic group over $F$. Since $H(F)$ embeds discretely in $H(\mathbb A_F)$, the adelic space $X_H= \frac{H(\mathbb A_F)}{H(F)}$ is Hausdorff. The Tamagawa number $\tau(H)$ is the volume of $X_H$ relative to a canonical (Tamagawa) measure \cite{tt}. The Tamagawa number theorem \cite{kottwitz, asok} (which was formerly a conjecture) states 
\begin{equation}\label{tam}\tau(H) = \frac{\#{\rm Pic}(H)_{\rm torsion}}{\#\Sha(H)}\end{equation} where ${\rm Pic}(H)$ is the Picard group and $\Sha(H)$ the Tate-Shafarevich set of $H/F$ (which measures the failure of the Hasse principle). Taking $H= SL_2$ over $\mathbb Q$ in (\ref{tam}) recovers Euler's result $$\zeta(2) = \frac{\pi ^2}{6}.$$

The above formulation (\ref{tam}) of the Tamagawa number theorem is due to T.~Ono \cite{ono, 44360} whose study of the behavior of $\tau$ under an isogeny   explains the presence of ${\rm Pic}(H)$, and reduces the semisimple case to the simply connected case.  The original form of the theorem (due to A.~Weil) is that $\tau(H) =1$ for split simply connected $H$. The Tamagawa number theorem (\ref{tam}) is valid, more generally, for any connected linear algebraic group $H$ over $F$. The case $H= \mathbb G_m$ becomes the Tate-Iwasawa  \cite{tate, iwasawa1} version of the analytic class number formula: the residue at $s=1$ of the zeta function $\zeta(F,s)$ is the volume of the (compact) unit idele class group $\mathbb J^1_F = {\rm Ker}(|-|: \frac{\mathbb I_F}{F^*} \to \mathbb R_{>0})$ of $F$. Here $|-|$ is the absolute value or norm map on $\mathbb I_F$.

\subsection*{Zagier extensions.} \cite{zagier} The $m$-Selmer group ${\rm Sel}_m(A/F)$ (for $m>0$) fits into an exact sequence\begin{equation}\label{selm} 0 \to \frac{A(F)}{m A(F)} \to {\rm Sel}_m(A/F) \to \Sha(A/F)_m\to 0.\end{equation}

D.~Zagier \cite[\S 4]{zagier} has pointed out that while the $m$-Selmer sequences (\ref{selm}) (for all $m >1$) cannot be induced by a sequence (an extension of $\Sha(A/F)$ by $A(F)$) 
$$ 0 \to A(F) \to ? \to \Sha(A/F) \to 0,$$ they can be induced by an exact sequence of the form 
\begin{equation}\label{don}
0 \to A(F) \to \mathcal A \to \mathcal S \to \Sha(A/F) \to 0\end{equation} and gave examples of such (Zagier) sequences. Combining (\ref{ext}) and (\ref{myst})  above provides the following natural Zagier sequence
$$ 0 \to A(F) \to A(\mathbb A_F) \to \frac{Y_A}{T(\mathbb A_F)} \to \Sha(A/F) \to 0.$$
 Write $A(\mathbb A_{\bar F})$ for the direct limit of the groups $A(\mathbb A_L)$ over all finite subextensions $F \subset L \subset \bar{F}$. The previous sequence discretized (neglect the topology) becomes $$0 \to A(F) \to A(\mathbb A_F) \to (\frac{A(\mathbb A_{\bar F})}{A(\bar{F})})^{\Gamma} \to \Sha(A/F) \to 0.$$ 
 
 \begin{remark}  (i) For an elliptic curve $E$ over $F$, Flach has indicated how to extract a canonical Zagier sequence via $\tau_{\ge 1}{\tau_{\le 2}}R\Gamma(S_{et}, \mathbb G_m)$ from any regular arithmetic surface $S\to \mathrm{Spec}~\mathcal O_F$ with $E= S \times_{\mathrm{Spec}~\mathcal O_F} \mathrm{Spec}~F$. 
 
  (ii) It is well known that the class group ${\rm Pic}(\mathcal O_F)$ is analogous to $\Sha(A/F)$ and the unit group $\mathcal O_F^{\times}$ is analogous to $A(F)$.  Iwasawa \cite[p.~354]{iwasawa2} proved that the compactness of $\mathbb J^1_F$  is equivalent to the two basic finiteness results of algebraic number theory: (i) ${\rm Pic}(\mathcal O_F)$ is finite; (ii) $\mathcal O_F^{\times}$ is finitely generated. His result provided a beautiful new proof of these two finiteness theorems. Bloch's result \cite[Theorem 1.10]{bloch} on the compactness of $X_A$ uses the Mordell-Weil theorem (the group $A(F)$ is finitely generated) and the non-degeneracy of the N\'eron-Tate pairing on $A(F) \times A^{\vee}(F)$ (modulo torsion). 
  \end{remark} 
  \begin{question} 
  {\it Can one define directly a space attached to $A/F$ whose compactness implies the Mordell-Weil theorem for $A$ and the finiteness of $\Sha(A/F)$?}
 \end{question}
 
 \subsection*{Acknowledgements.}  I thank C.~Deninger,  M.~Flach, S.~Lichtenbaum, J.~Milne, J.~Parson, J.~Rosenberg, L.~Washington, and B.~Wieland for interest and encouragement and the referee for helping correct some inaccuracies in an earlier version of this note.  I  have been inspired by the ideas of Bloch, T.~Ono (via  B.~Wieland \cite{44360}) on Tamagawa numbers, and D.~Zagier  on the Tate-Shafarevich group \cite{zagier}. I am indebted to Ran Cui for alerting me to Wieland's ideas \cite{44360}.

\def\cprime{$'$} \def\cprime{$'$}

\end{document}